\documentclass[11pt,twoside,reqno]{article}
\usepackage{fullpage}
\usepackage{amssymb,amsfonts,amsmath,amsthm,bm}

\usepackage{ stmaryrd}
\theoremstyle{definition}
	\newtheorem{thm}{Theorem}
	\newtheorem{lem}[thm]{Lemma}

\usepackage{enumerate}
\usepackage{enumitem}
\usepackage{hyperref}
\usepackage{textcomp}
\usepackage{color}

\usepackage[mathscr]{eucal}
\usepackage{bbm}

\usepackage{hyperref}
\usepackage[dvipsnames]{xcolor}
\newcommand\myshade{85}
\colorlet{mylinkcolor}{red}
\colorlet{mycitecolor}{blue}
\colorlet{myurlcolor}{Aquamarine}
\hypersetup{
	linkcolor  = mylinkcolor,
	citecolor  = mycitecolor,
	urlcolor   = myurlcolor!\myshade!black,
	colorlinks = true,
}

\newcommand{\lhat}{{\widehat{\lambda}}}

\newcommand{\Mhat}{{\widehat{M}}}

\newcommand{\Lhat}{{\widehat{\Lambda}}}

\newcommand{\ds}{\displaystyle}

\numberwithin{equation}{section}
\allowdisplaybreaks

\newcommand{\ignore}[1]{}

\usepackage{array}

\begin{document}
\thispagestyle{empty}
\title{A refinement of and a companion to MacMahon's partition identity}

\author{Matthew C. Russell\footnote{University of Illinois Urbana-Champaign. \url{mcr39@illinois.edu}}}

\maketitle

\begin{abstract}
We provide a refinement of MacMahon's partition identity on sequence-avoiding partitions, and use it to produce another mod 6 partition identity. In addition, we show that our technique also extends to cover Andrews's generalization of MacMahon's identity. Our proofs are bijective in nature, exploiting a theorem of Xiong and Keith.
\end{abstract}

\section{Introduction}

In volume 2 of his seminal textbook {\it{Combinatory Analysis}} from 1916~\cite{MacBook}, Major Percy MacMahon proved the following mod 6 partition identity:
\begin{thm}[MacMahon]
Let $n$ be a nonnegative integer.
\begin{itemize}\label{MacM}
    \item Let $A_1(n)$ be the number of partitions of $n$ into parts congruent to $0,$ $2,$ $3,$ or $4 \pmod 6$.
    \item Let $A_2(n)$ be the number of partitions of $n$ where no part occurs exactly once.
    \item Let $A_3(n)$ be the number of partitions of $n$ where no consecutive integers appear as parts, and all parts are at least 2.
\end{itemize}
Then, $A_1(n)=A_2(n)=A_3(n)$.
\end{thm}
Note that the equality $A_2(n)=A_3(n)$ is relatively trivial, as can be seen by taking the conjugates of the partitions counted by $A_2(n)$ or $A_3(n)$.  Partitions satisfying the condition for $A_3(n)$ that no consecutive integers appear as parts are sometimes described as ``sequence-avoiding partitions''.

A half-century later, this was generalized by George E. Andrews~\cite{And}:

\begin{thm}[Andrews]\label{thm:And} Let $r$ be a positive integer. The number of partitions of $n$ in which any part with odd multiplicity must appear at least $2r+ 1$ times equals the number of partitions of $n$ where all parts must be even or congruent to $2r+ 1  \pmod {4r+ 2}$. \end{thm}

The proofs of MacMahon and Andrews involved relatively straightforward manipulation of generating functions, as did the proofs of M.~V.~Subbarao~\cite{Sub} in his further generalizations of these identities. However, it was not until 2007 that the first bijective proof of MacMahon's identity was provided (by Andrews, Henrik Eriksson, Fedor Petrov, and Dan Romik~\cite{AEPR}). Subsequently, in 2014, Shishuo Fu and James A. Sellers~\cite{FS} gave an alternative bijective proof (also covering the extensions of Andrews and of Subbarao, and extensions of their own). Beaullah Mugwangwavari and Darlison Nyirenda~\cite{MN,NM} have recently given new bijective proofs of many of these identities.

The bijection of Fu and Sellers~\cite{FS} is the one that is most pertinent for our purposes. Their key step was to interpret $0,$ $2,$ $3,$ or $4 \pmod 6$ as $0 \pmod 2$ or $3 \pmod 6$. Then, even parts $2j$ were replaced with two copies of $j$, while parts that are 3 times an odd integer were sent to partitions where every part occurred exactly three times, following Euler's identity. Instead, our basic plan will be to think of $0,$ $2,$ $3,$ or $4 \pmod 6$ as $0 \pmod 3$ or $\pm2 \pmod 6$. While the parts $0 \pmod 3$ are easy enough to deal with, working with the parts $\pm2 \pmod 6$ is a more involved process. 

Our main tool (presented as Theorem~\ref{XKThm} below) is a result due to Xinhua Xiong and William Keith~\cite{XK}. The process is similar to that of the work of Shashank Kanade, Debajyoti Nandi, and the author~\cite{KNR}. It will allow us to obtain a refinement of Theorem~\ref{MacM}, shown below as Theorem~\ref{thm:orig}. This refinement can then be used to prove companion results, including a three-color partition identity, and the following companion result to Theorem~\ref{MacM}:
\begin{thm}\label{IntroFriend}
Let $n$ be a nonnegative integer.
\begin{itemize}
    \item Let $C_1(n)$ be the number of partitions of $n$ into parts congruent to $0,$ $1,$ $3,$ or $5 \pmod 6$.
    \item Let $C_2(n)$ be the number of partitions of $n$ that satisfy the following conditions:
    \begin{itemize}
        \item If adjacent parts differ by exactly 1, the smaller part cannot be $\equiv 1 \pmod 3$. (Equivalently, if adjacent parts differ by exactly 1, their sum cannot be $\equiv 0 \pmod 3$).
        \item If adjacent parts differ by exactly 2, the smaller part must be $\equiv 2 \pmod 3$. (Equivalently, if adjacent parts differ by exactly 2, their sum must be $\equiv 0 \pmod 3$).
        \item No parts are equal to 2.
    \end{itemize}
\end{itemize}
Then, $C_1(n)=C_2(n)$.
\end{thm}
To illustrate this theorem, we observe that $C_1(11)=C_2(11)=15$. The partitions counted by $C_1(11)$ are $11$, $9+1+1$, $7+3+1$, $7+1+1+1+1$, $6+5$, $6+3+1+1$, $6+1+1+1+1+1$, $5+5+1$, $5+3+3$, $5+3+1+1+1$, $5+1+1+1+1+1+1$, $3+3+3+1+1$, $3+3+1+1+1+1+1$, $3+1+1+1+1+1+1+1+1$, and $1+1+1+1+1+1+1+1+1+1+1$, while the partitions counted by $C_2(11)$ are $11$, $10+1$, $9+1+1$, $8+3$, $8+1+1+1$, $7+4$, $7+1+1+1+1$, $6+5$, $6+1+1+1+1+1$, $5+5+1$, $5+1+1+1+1+1+1$, $4+4+3$, $4+4+1+1+1$, $4+1+1+1+1+1+1+1$, and $1+1+1+1+1+1+1+1+1+1+1$.

We then conclude by applying the same technique to obtain a similar refinement of Theorem~\ref{thm:And}.

\section{Preliminaries}

Let $n$ be a non-negative integer. A {\it partition} of $n$ is a list of integers $\left(\lambda_1,\lambda_2,\dots,\lambda_k\right)$ such that $ \lambda_1+\lambda_2+\cdots+\lambda_k = n$ and  $\lambda_1 \ge \lambda_2 \ge \cdots \ge \lambda_k \ge 1$. We will frequently write this as $\lambda_1+\lambda_2+\cdots+\lambda_k$. The {\it weight} of a partition $\lambda$, $|\lambda|$, is the sum of the parts. We also will use $|\lambda|_{i(m)}$ to represent the sum of the parts of $\lambda$ that are congruent to $i\pmod m$. As an example, if $\lambda=9+8+8+5+2+1$, then $|\lambda|_{2(6)}=8+8+2=18$. The multiplicity of a part $i$ in a partition $\lambda$ is the number of times it occurs: the part 8 in the preceding example has multiplicity 2.

The length of a partition $\lambda$, $\ell(\lambda)$, is the number of parts of $\lambda$. The {\it conjugate} of a partition $\left(\lambda_1,\lambda_2,\dots,\lambda_k\right)$ is a second partition $\left(\mu_1,\mu_2,\dots,\mu_j\right)$, where $\mu_i$ is chosen to equal the number of parts of $\mu$ that are greater than or equal to $i$. If $m$ is a positive integer, a partition in which no parts are congruent to $0 \pmod m$ is called $m$-regular.

We will occasionally want to present our results in generating function form as formal power series. For example, if $\Lambda$ is the set of all sequence-avoiding partitions with no ones, then Theorem~\ref{MacM} is equivalent to the statement
\begin{equation}
\sum_{\lambda \in \Lambda} q^{|\lambda|} = \frac{1}{\left(q^2,q^3,q^4,q^6;q^6\right)_\infty}.
\end{equation}
Here, we employ the standard $q$-series notation $\left(a,q\right)_\infty = \prod_{j\ge 0} \left(1-aq^j\right)$, and \begin{equation}
\left(a_1;q\right)=\left(a_1,a_2,\dots,a_k;q\right)\left(a_2;q\right)\cdots\left(a_k;q\right).
\end{equation}

Following Xiong and Keith~\cite{XK}, we define the {\it $m$-length type} of a partition to be the $(m-1)$-tuple $\left(\alpha_1,\alpha_2,\dots,\alpha_{m-1}\right)$, where there are $\alpha_i$ parts congruent to $i \pmod m$. Furthermore, we define the {\it $m$-alternating sum type} of a partition to be the $(m-1)$-tuple $\left(M_1-M_2,M_2-M_3,\dots,M_{m-1}-M_{m}\right)$, where $M_i$ is the sum of all parts in the partition whose index is congruent to $i \pmod m$. Note that we define the $m$-length type of a partition even when it is not $m$-regular, and we define the $m$-alternating sum type of a partition even if it has parts that occur at least $m$ times.

We can now present the following remarkable theorem of Xiong and Keith~\cite{XK}, which itself extends a theorem of Igor Pak and Alexander Postnikov~\cite{PP}:

\begin{thm}[Xiong and Keith]\label{XKThm}
Consider a modulus $m$ and a nonnegative integer $n$. The number of $m$-regular partitions of $n$ with $m$-length type $\left(\alpha_1,\alpha_2,\dots,\alpha_{m-1}\right)$ equals the number of partitions of $n$ with no part occurring $m$ or more times with $m$-alternating sum type $\left(\alpha_1,\alpha_2,\dots,\alpha_{m-1}\right).$
\end{thm}

The proof of Xiong and Keith's theorem uses a modification of a bijection of Dieter Stockhofe~\cite{Sto}. A translation of Stockhofe's thesis~\cite{Sto} from German into English was helpfully provided by Keith as an appendix to his thesis~\cite{KeiThe}. Fu, Dazhao Tang, and Ae Ja Yee~\cite{FTY} have recently refined Theorem~\ref{XKThm}. Ya Gao and Xiong~\cite{GX} provided an analytic proof of certain cases of Theorem~\ref{XKThm}, and Isaac Konan~\cite{Kon} generalized Theorem~\ref{XKThm} to certain families of colored partitions. Cristina Ballantine and Amanda Welch~\cite{BW1,BW2} have used Theorem~\ref{XKThm} to prove certain Beck-type identities.

We will also require the following lemma of Xiong and Keith. 

\begin{lem}[Xiong and Keith] \label{lem:conj}
The conjugates of partitions $\lambda$ with $m$-alternating sum type $\left(s_1,\dots,s_{m-1}\right)$ are precisely those partitions of $m$-length type $\left(s_1,\dots,s_{m-1}\right)$.
\end{lem}

Here is one more lemma that will be useful for us:

\begin{lem}\label{lem:dup2}
Suppose that $\lambda$ is a partition with $m$-alternating sum type $\left(\Sigma_1,\Sigma_2,\dots,\Sigma_{m-1}\right)$. Then, if we form a new partition $\lhat$ with weight $|\lhat|=(m-1)|\lambda|$ by replacing each part $i$ of $\lambda$ with $m-1$ copies of $i$, then $\lhat$ has $m$-alternating sum type $\left(\Sigma_{m-1},\Sigma_{m-2},\dots,\Sigma_1\right)$.
\end{lem}
\begin{proof}
For $1\le i \le m$, let $M_i$ be the sum of all parts of $\lambda$ with indices congruent to $i\pmod m$, and let $\Mhat_i$ be the sum of all parts of $\lhat$ with indices congruent to $i\pmod m$. Observe that $\Mhat_i=|\lambda|-M_{m-i+1}$. Thus, for $1\le i \le m-1$, $\Mhat_i-\Mhat_{i+1}=\left(|\lambda|-M_{m-i+1}\right)-\left(|\lambda|-M_{m-i}\right)=M_{m-i}-M_{m-i+1}=\Sigma_{m-i}$.
\end{proof}

\section{Main results and proofs}

\subsection{A refinement of Theorem~\ref{MacM}}

Here is our refinement of Theorem~\ref{MacM}:

\begin{thm} \label{thm:orig}
Let $n,m_1,m_2$ be nonnegative integers.
\begin{itemize}
    \item Let $B_1\left(m_1,m_2,n\right)$ be the number of partitions of $n$ into parts congruent to $0,$ $2,$ $3,$ or $4 \pmod 6$, where there are $m_1$ parts congruent to $2 \pmod 6$ and $m_2$ parts congruent to $4 \pmod 6$.
    \item Let $B_2\left(m_1,m_2,n\right)$ be the number of partitions of $n$ where no consecutive integers appear as parts and all parts are at least 2, where there are $m_2$ parts congruent to $1 \pmod 3$ and $m_1$ parts congruent to $2 \pmod 3$.
\end{itemize}
Then, $B_1\left(m_1,m_2,n\right)=B_2\left(m_1,m_2,n\right)$.

Equivalently, we have the following generating function representation:
\begin{equation}
\sum_{\lambda \in \Lambda} A^{\#_{2,3}(\lambda)} C^{\#_{1,3}(\lambda)} q^{|\lambda|} = \frac{1}{\left(q^3;q^3\right)_{\infty}\left(Aq^2;q^6\right)_\infty\left(Cq^4;q^6\right)_\infty}
\end{equation}
where $\Lambda$ is the set of all sequence-avoiding partitions with no parts equal to 1, the letters $A$ and $C$ keep track of the number of parts congruent to $2 \pmod 3$ and $1 \pmod 3$, and $\#_{1,3}(\lambda)$ and $\#_{2,3}(\lambda)$ are the number of parts of $\lambda$ congruent to $1 \pmod 3$ and $2 \pmod 3$, respectively.
\end{thm}
\begin{proof}
Let $\lambda$ be a partition of $n$ into parts congruent to $0 \pmod 3$ or $2$ or $4 \pmod 6$ with $m_1$ parts congruent to $2 \pmod 6$ and $m_2$ parts congruent to $4 \pmod 6$. We will now demonstrate how to change $\lambda$ into a sequence-avoiding partition $\tau$ with no parts equal to 1 with $m_2$ parts congruent to $1 \pmod 3$ and $m_1$ parts congruent to $2 \pmod 3$.

Let $\mu$ be the partition formed by taking only the parts of $\lambda$ congruent to $2$ or $4 \pmod 6$ and dividing them by 2. We can see that $\mu$ is a 3-regular partition of length type $\left(m_1,m_2\right)$, and $|\mu|=\ds\frac{|\lambda|_{2(6)}+|\lambda|_{4(6)}}2$. We use the bijection of Theorem~\ref{XKThm} to send this to a partition $\nu$ with at most two occurrences of each part with $3$-alternating sum type $\left(m_1,m_2\right)$. Next, we duplicate each part of $\nu$ (replace each part $i$ by $i+i$) to obtain a partition $\pi$ in which each part occurs exactly 2 times or exactly 4 times. Observe that $|\pi|=|\lambda|_{2(6)}+|\lambda|_{4(6)}$. By Lemma~\ref{lem:dup2}, we see that the $3$-alternating sum type of $\pi$ is $\left(m_2,m_1\right)$. Next, we create partition $\rho$ by taking the leftover parts of $\lambda$ that are congruent to $0\pmod 3$, replacing each part $3i$ with $i+i+i$, and combining these parts with $\pi$, so $|\rho|=|\lambda|_{2(6)}+|\lambda|_{4(6)}+|\lambda|_{0(3)}=|\lambda|$. Moreover, $\rho$ is a partition in which no part occurs once --- but other than that, there are no restrictions on the multiplicities of parts. Additionally, $\rho$ still has $3$-alternating sum type $\left(m_2,m_1\right)$. To conclude our map, we take the conjugate of $\rho$ to obtain a partition $\tau$. By Lemma~\ref{lem:conj}, the $3$-length type of $\tau$ is $\left(m_2,m_1\right)$, which means it has $m_2$ parts congruent to $1 \pmod 3$ and $m_1$ parts congruent to $2 \pmod 3$, as desired. Since $\rho$ has no parts with multiplicity 1, $\tau$ is a sequence-avoiding partition with no parts equal to 1.

It is easily seen that the map above is reversible. For a sequence-avoiding partition of $n$ without 1s, first take its conjugate, which will be a partition where no parts occur exactly once. Then, for this partition, for each part $i$ with multiplicity $\ge 3$, delete subpartitions of the form $i+i+i$, until a partition is obtained where parts that appear only have multiplicity 2 or multiplicity 4. Delete every other part to obtain a partition where each part occurs at most twice, and run the bijection of Theorem~\ref{XKThm} in the other direction to obtain a 3-regular partition. Multiply each of these parts by 2, and then, for each deleted triple of parts $i+i+i$, restore a part $3i$.
\end{proof}

The following table illustrates Theorem~\ref{thm:orig} in the case of $n=15$. The first column contains the seven different possibilities of $\left(m_1,m_2\right)$, while the second and third columns contain the partitions of 15 corresponding to those values of $\left(m_1,m_2\right)$ for $B_1\left(m_1,m_2,15\right)$ and $B_2\left(m_1,m_2,15\right)$, respectively.
\begin{center}
\begin{tabular}{ |c | c | c | } \hline
 $\left(m_1,m_2\right)$ & $B_1\left(m_1,m_2,15\right)$ & $B_2\left(m_1,m_2,15\right)$ \\  \hline
 $(0,0)$ & 15  & 15 \\  
         & 12+3  & 12+3 \\  
         & 9+6   & 9+6 \\  
         & 9+3+3  & 9+3+3 \\  
         & 6+6+3  & 6+6+3  \\ 
         & 6+3+3+3  & 6+3+3+3  \\
         & 3+3+3+3+3  & 3+3+3+3+3 \\\hline
$(1,1)$  & 10+3+2 &  13+2 \\
         & 9+4+2 &  11+4 \\
         & 8+4+3 & 10+5 \\
         & 6+4+3+2 & 9+4+2 \\
         & 4+3+3+3+2 & 7+5+3 \\ \hline
$(3,0)$  & 9+2+2+2 & 11+2+2 \\
         & 8+3+2+2 & 9+2+2+2 \\
         & 6+3+2+2+2 & 8+5+2 \\
         & 3+3+3+2+2+2 & 5+5+5 \\\hline
$(0,3)$  & 4+4+4+3 & 7+4+4 \\ \hline
$(2,2)$  & 4+4+3+2+2 & 7+4+2+2 \\ \hline
$(4,1)$  & 4+3+2+2+2+2 & 7+2+2+2+2 \\ \hline
$(6,0)$  & 3+2+2+2+2+2+2 & 5+2+2+2+2+2 \\ \hline
\end{tabular}
\end{center}

We also illustrate the process of mapping $\lambda$ to $\tau$ for $\lambda= 39+38+34+28+26+26+18+16+3+2$, which is a partition of $n=230$ into parts congruent to $0,$ $2,$ $3,$ or $4 \pmod 6$, where there are 4 parts congruent to $2 \pmod 6$ and 3 parts congruent to $4 \pmod 6$. The process for obtaining $\mu,$ $\pi$, $\rho$, and $\tau$ using the process given in the above proof should be straightforward.  In this paper, we do not give the algorithm for mapping $\nu$ to $\pi$. However, but we have chosen our example so that $\nu$ and $\pi$ match the sample $\lambda$ and $\mu$ given in section 2 of~\cite{FTY}, for which Fu, Tang, and Yee have given a thorough demonstration of the required bijection. The interested reader may want to work through the example in that paper.
\begin{align*}
\lambda &= 39+38+34+28+26+26+18+16+3+2 \\ 
\mu &= 19+17+14+13+13+8+1 \\
\nu &= 11+10+9+9+8+8+6+5+5+4+4+2+2+1+1  \\ 
\pi &= 11+11+10+10+9+9+9+9+8+8+8+8+6+6 \\
& \quad +5+5+5+5+4+4+4+4+2+2+2+2+1+1+1+1\\
\rho &= 13+13+13+11+11+10+10+9+9+9+9+8+8+8+8+6+6+6+6+6 \\
& \quad +5+5+5+5+4+4+4+4+2+2+2+2+1+1+1+1+1+1+1\\
\tau &= 39+32+28+28+24+20+15+15+11+7+5+3+3
\end{align*}
As desired, $\tau$ is also a partition of $n=230$ with 3 parts congruent to $1 \pmod 3$ and 4 parts congruent to $2 \pmod 3$. Observe that $\mu$ has 3-length type $(4,3)$, and so $\nu$ has 3-alternating sum type $(4,3).$ Also, $\pi$ has 3-alternating sum type $(3,4)$.

\subsection{A three-color identity and Theorem~\ref{IntroFriend}}

But we can prove even more from Theorem~\ref{thm:orig}!

Suppose we consider partitions into three colors, $a$, $b$, and $c$, where the parts are arranged in the following order:
\begin{equation}
1_a < 1_b < 1_c < 2_a < 2_b < 2_c < 3_a < \cdots
\end{equation}

Let us take the sequence-avoiding partitions from Theorem~\ref{thm:orig}, and map parts of the form $3k$ to $k_b$, parts of the form $3k-1$ to $k_a$, and parts of the form $3k+1$ to $k_c$. This sends $\Lambda$ to the set of three-colored partitions $\Lambda^\star$ without occurrences of consecutive parts. That is, we forbid $j_a + j_b$, $j_b+j_c$, and $j_c + (j+1)_a$ from appearing in our partitions.

Let $|\lambda|$ be the sum of the parts of $\lambda$, ignoring color. Let $\#_a(\lambda)$ be the number of parts of $\lambda$ with color $a$, and define $\#_c(\lambda)$ similarly. Using these definitions, we then have the following theorem:

\begin{thm}
\begin{equation}
\sum_{\lambda \in \Lambda^\star} A^{\#_a(\lambda)} C^{\#_c(\lambda)} q^{|\lambda|} = \frac{1}{(q;q)_{\infty}\left(Aq;q^2\right)_\infty\left(Cq;q^2\right)_\infty}
\end{equation}
\end{thm}

We can now use this three-colored partition identity to produce Theorem~\ref{IntroFriend}. For a partition in $\Lambda^\star$, simply send parts of the form $k_a$ to $3k-2$, parts of the form $k_b$ to $3k$, and parts of the form $k_c$ to $3k+2$. This produces the ordering 
\begin{equation}
1 <' 3 <' 5 <' 4 <' 6 <' 8 <' 7 <'  9 <' 11 <' 10 \cdots,
\end{equation}
where 2 does not appear, and adjacent parts under the above ordering are not allowed to both occur. (For example, under these rules, a partition is not allowed to contain $8+6$, but it may contain $7+6$.) Then, we can see that Theorem~\ref{IntroFriend} directly follows from this dilation. In fact, we have the following refinement:

\begin{thm}
Let $n,m_1,m_2$ be nonnegative integers.
\begin{itemize}
    \item Let $C_1\left(m_1,m_2,n\right)$ be the number of partitions of $n$ into parts congruent to $0,$ $1,$ $3,$ or $5 \pmod 6$, where there are $m_1$ parts congruent to $1 \pmod 6$ and $m_2$ parts congruent to $5 \pmod 6$.
    \item Let $C_2\left(m_1,m_2,n\right)$ be the number of partitions of $n$ that satisfy the following conditions:
    \begin{itemize}
        \item If adjacent parts differ by exactly 1, the smaller part cannot be $\equiv 1 \pmod 3$. (Equivalently, if adjacent parts differ by exactly 1, their sum cannot be $\equiv 0 \pmod 3$).
        \item If adjacent parts differ by exactly 2, the smaller part must be $\equiv 2 \pmod 3$. (Equivalently, if adjacent parts differ by exactly 2, their sum must be $\equiv 0 \pmod 3$).
        \item No parts are equal to 2.
        \item There are $m_1$ parts congruent to $1 \pmod 3$ and $m_2$ parts congruent to $2 \pmod 3$.
    \end{itemize}
\end{itemize}
Then, $C_1\left(m_1,m_2,n\right)=C_2\left(m_1,m_2,n\right)$.
\end{thm}

\subsection{A refinement of Theorem~\ref{thm:And}}

We conclude with the corresponding refinement of Theorem~\ref{thm:And}:
\begin{thm} \label{thm:gen}
Let $r$ be a positive integer, and let $n,m_1,\dots,m_{2r}$ be nonnegative integers.
\begin{itemize}
    \item Let $D_1\left(m_1,m_2,\dots,m_{2r},n\right)$ be the number of partitions of $n$ where all parts must be even or congruent to $2r+ 1  \pmod {4r+ 2}$, where, for $1\le i \le 2r$, there are $m_i$ parts congruent to $2i\pmod{4r+2}$.
    \item Let $D_2\left(m_1,m_2,\dots,m_{2r},n\right)$ be the number of partitions of $n$ where the difference between consecutive parts is not in the set $\{1,3,\dots,2r-1\}$ and the smallest odd part is at least $2r+1$, where, for $1\le i \le 2r$,  there are $m_i$ parts congruent to $2r-i+1\pmod{2r+1}$.
\end{itemize}
Then, $D_1\left(m_1,m_2,\dots,m_{2r},n\right)=D_2\left(m_1,m_2,\dots,m_{2r},n\right)$.
\end{thm}
\begin{proof}
Let $\lambda$ be a partition of $n$ into parts congruent to $0 \pmod {2r+1}$ or $2,4,\dots,4r\pmod{4r+2}$, where, for $1\le i \le 2r$, there are $m_i$ parts congruent to $2i\pmod{4r+2}$. 
We will now demonstrate how to change $\lambda$ into a partition $\tau$ corresponding to the conditions of $D_2\left(m_1,m_2,\dots,m_{2r},n\right)$.

Let $\mu$ be the partition formed by taking only the parts of $\lambda$ congruent to $2,4,\dots,4r\pmod{4r+2}$, and dividing them by 2. We can see that $\mu$ is a $(2r+1)$-regular partition of length type $\left(m_1,m_2,\dots,m_{2r}\right)$, and $|\mu|=\ds\frac{|\lambda|_{2(4r+2)}+\cdots+|\lambda|_{4r(4r+2)}}2$. We use the bijection of Theorem~\ref{XKThm} to send this to a partition $\nu$ with at most $2r$ occurrences of each part with $(2r+1)$-alternating sum type $\left(m_1,m_2\dots,m_{2r}\right)$. Next, we duplicate each part of $\nu$ to obtain a partition $\pi$ in which the multiplicity of each part that occurs is in $\{2,4,\dots,4r\}$. Note that $|\pi|=|\lambda|_{2(4r+2)}+\cdots+|\lambda|_{4r(4r+2)}$. By Lemma~\ref{lem:dup2}, we see that the $(2r+1)$-alternating sum type of $\rho$ is $\left(m_{2r},m_{2r-1},\dots,m_1\right)$. Next, we create partition $\rho$ by taking the leftover parts of $\lambda$ that are congruent to $0\pmod {2r+1}$, replacing each part $(2r+1)i$ with $2r+1$ copies of $i$, and combining these parts with $\pi$, so $|\rho|=|\lambda|_{2(4r+2)}+\cdots+|\lambda|_{4r(4r+2)}+|\lambda|_{0(2r+1)}=|\lambda|$. Moreover, $\rho$ is a partition in which parts are forbidden to have multiplicity in the set $\{1,3,5,\dots,2r-1\}$. Additionally, $\rho$ still has $(2r+1)$-alternating sum type $\left(m_{2r},m_{2r-1},\dots,m_1\right)$. To conclude our map, we take the conjugate of $\rho$ to obtain a partition $\tau$. By Lemma~\ref{lem:conj}, the $\left(m_{2r},m_{2r-1},\dots,m_1\right)$-length type of $\tau$ is $\left(m_{2r},m_{2r-1},\dots,m_1\right)$, as desired. Finally, since $\rho$ has no parts
with multiplicity in the set $\{1,3,5,\dots,2r-1\}$, the difference between consecutive parts of $\tau$ is not in the set $\{1,3,\dots,2r-1\}$ and the smallest odd part of $\tau$ is at least $2r+1$.

As before, it is straightforward to reverse the above map.
\end{proof}

We can convert this to a colored partition identity with $2r+1$ colors, where the parts are ordered
\begin{equation}
1_1 < 1_2 < 1_3 < \cdots <1_{2r+1} < 2_1 < \cdots.
\end{equation}
This can be achieved by mapping, for $k\ge 0$,
\begin{align}
&(2r+1)(k-1)+2i \mapsto k_i, \quad \quad  1 \le i \le 2r \\
&(2r+1)k \mapsto k_{2r+1} 
\end{align}
Then, the difference conditions for $D_2$ become forbidding the following pairs of parts:
\begin{itemize}
\item $j_i$ and ${(j+1)}_\ell$, for $i\le 2r$ and $\ell<i$
\item $j_i$ and $j_{2r+1}$, for $i\le 2r$.
\end{itemize}

If $\Lhat$ is the set of colored partitions satisfying these conditions, and, for $1\le i \le 2r$, $A_i$ is a variable to keep track of the number of parts of $\lambda$ of color $i$ (represented by $ {\#_i(\lambda)} $), then we have

\begin{equation}
\sum_{\lambda \in \Lhat} q^{|\lambda|} \prod_{i=1}^{2r} A_i ^ {\#_i(\lambda)}  =  \frac{1}{(q;q)_{\infty}\prod_{i=1}^{2r} \left(A_i q;q^2\right)_\infty}.
\end{equation}

\section{Acknowledgements}

The author benefited greatly from conversations with George E. Andrews, Shashank Kanade,  William Keith, Karl Mahlburg, and Jan Russell.

\bibliographystyle{plain}
\bibliography{Mac}

\end{document}